\documentclass[12pt,a4paper]{amsart}
\usepackage{amsfonts,amsmath,amssymb}
\usepackage{hyperref}
\usepackage{latexsym,fullpage,amsfonts,amssymb,amsmath,amscd,graphics,epic}
\usepackage[all]{xy}
\usepackage{amssymb,amsbsy,amsthm,amsxtra}
\usepackage[usenames]{color}
\usepackage{amscd}
\usepackage{amsthm}
\usepackage{amsfonts}
\usepackage{amssymb}
\usepackage{mathrsfs}
\usepackage{url}
\usepackage{bbm}
\usepackage{wasysym}
\usepackage{enumitem}
\usepackage{framed}

\theoremstyle{plain}
\newtheorem*{theorem*}{Theorem}
\newtheorem*{remark*}{Remark}
\newtheorem*{example*}{Example}
\newtheorem{lemma}{Lemma}[subsection]
\newtheorem{proposition}[lemma]{Proposition}
\newtheorem{corollary}[lemma]{Corollary}
\newtheorem{theorem}[lemma]{Theorem}

\newtheorem*{conjecture*}{Conjecture}

\newtheorem{introtheorem}{Theorem}

\theoremstyle{definition}
\newtheorem{definition}[lemma]{Definition}

\newtheorem{example}[lemma]{Example}

\theoremstyle{remark}
\newtheorem{remark}[lemma]{Remark}

\newtheorem{notation}[lemma]{Notation}

\newcommand{\Hom}{\operatorname{Hom}}

\newcommand{\triv}{{\mathbbm 1}}

\newcommand{\id}{\operatorname{Id}}

\newcommand{\Ker}{\operatorname{Ker}}

\newcommand{\F}{{\mathcal F}}

\newcommand{\rk}{{\operatorname{rk}}}
\newcommand{\sdim}{{\operatorname{sdim}}}

\newcommand{\C}{{\mathbb C}}
\newcommand{\Z}{{\mathbb Z}}

\newcommand{\eps}{{\varepsilon}}

\newcommand{\lam}{{\lambda}}

\newcommand{\gl}{\mathfrak{gl}}
\newcommand{\abs}[1]{\left|{#1}\right|}

\newcommand{\p}{\mathfrak{p}}

\newcommand{\Inna}[1]{\begin{framed} {\tt{\color{blue}{#1}}} \end{framed}} 
\newcommand{\VeraA}[1]{{#1}}

\newcommand{\comment}[1]{}

\def\quotient#1#2{%
    \raise1ex\hbox{$#1$}\Big/\lower1ex\hbox{$#2$}%
}

\begin{document}

\date{\today}
\title{Kac-Wakimoto conjecture for the periplectic Lie superalgebra}
 \author{Inna Entova-Aizenbud, Vera Serganova}
\address{Inna Entova-Aizenbud, Dept. of Mathematics, Ben Gurion University,
Beer-Sheva,
Israel.}
\email{entova@bgu.ac.il}
\address{Vera Serganova, Dept. of Mathematics, University of California at
Berkeley,
Berkeley, CA 94720.}
\email{serganov@math.berkeley.edu}
%
%
\begin{abstract}
 We prove the Kac-Wakimoto conjecture for the periplectic Lie superalgebra $\p(n)$, stating that any simple module lying in a block of non-maximal atypicality has superdimension zero.
\end{abstract}

\maketitle
\setcounter{tocdepth}{3}

\section{Introduction}
\subsection{}
Consider a complex vector superspace $V$, and let $\C^{0|1}$ be the odd one-dimensional vectors superspace. 

The (complex) periplectic Lie superalgebra $\p(V)$ is the Lie superalgebra of 
endomorphisms of a complex vector superspace $V$ possessing a non-degenerate 
symmetric form $\omega: \VeraA{S^2 V} \to \C^{0|1}$ (this form is also referred to as an 
``odd form''). An example of such superalgebra is $\p(n)= \p(\C^{n|n})$ for $V = \C^{n|n}$, where 
$\omega_n: \C^{n|n} \otimes \C^{n|n} \to \C^{0|1}$ pairing the even and odd 
parts of the vector superspace $\C^{n|n}$.

The periplectic Lie superalgebras 
\VeraA{has} an interesting non-semisimple representation theory; some results 
on the category $\F_n$ of finite-dimensional integrable representations 
of $\p(n)$ can be found in \cite{BDE:periplectic, Chen:periplectic, Coulembier:Brauer, DLZ:fft_ortho, Gor:center, Moon, Ser:periplectic}.

In \cite{BDE:periplectic}, the blocks of the category $\F_n$ were classified: it was shown that (up to change of pairity) the blocks can be numbered by integers $-n, -n+2, -n+4, \ldots, n-4, n-2, n$, with the trivial representation sitting in block number $n (\mathrm{mod} 2)$. We denote this blocks by $\F^k_n$, $k\in \{-n, -n+2, \ldots, n-2, n\}$.

In this article, we prove the following version of the Kac-Wakimoto conjecture:

\begin{introtheorem}\label{introthrm:KW}
 We have $\sdim M =0$ if $M \in \F^k_n$, $ k\neq 0, \pm 1$, where $\sdim$ denotes the superdimension ($\sdim M = \dim M_{\bar{0}} - \dim M_{\bar{1}}$).
\end{introtheorem}

The main ingredients in the proof of this theorem, are the translation functors acting on $\F_n$, and the Duflo-Serganova functor $DS:\F_n \to \F_{n-2}$. The translation functors are direct summands of the functor $-\otimes V$, whose action on the blocks $\F_n$ was obtained in \cite{BDE:periplectic}; the functor $DS:\F_n \to \F_{n-2}$ is a tensor functor preserving dimension, which allows us to reduce the problem of computing dimensions in $\F_n$ to a similar problem in $\F_{n-2}$.

We also prove the following statement:

\begin{introtheorem}\label{introthrm:blocks}
 Let $M \in \F_n$ be an object lying in a certain block as described in Section \ref{sssec:blocks} and \cite{BDE:periplectic}. 
 
 \begin{enumerate}
  \item The object $M^*$ also lies in the same block of $\F_n$.
  \item We have a natural isomorphism $$\Theta_k(M^*) \cong \Pi \left(\Theta_{-k}(M)\right)^*$$ where $\Theta_k$ is the $k$-th translation functor on $\F_n$ (see Definition \ref{def:transl_functors}). 
 \end{enumerate}

\end{introtheorem}
\subsection{Acknowledgements}
I.E.-A. was supported by the ISF grant no. 711/18. V.S. was supported by NSF grant 1701532.
\section{Preliminaries}\label{sec:prelim}
\subsection{General}

Throughout this paper, we will work over the base field $\mathbb C$, and all 
the categories considered will be $\C$-linear.

A {\it 
vector superspace} will be defined as a $\mathbb{Z}/2\mathbb{Z}$-graded vector 
space $V=V_{\bar 0}\oplus V_{\bar 1}$. The {\it parity} of a homogeneous vector 
$v \in V$ will be denoted by $p(v) \in \mathbb{Z}/2\mathbb{Z}=\{\bar 0, \bar 
1\}$ (whenever the notation $p(v)$ appears in formulas, we always assume that 
$v$ is homogeneous).

\subsection{Tensor categories}\label{ssec:notn:SM}
In the context of symmetric monoidal (SM) categories, we will denote by $\triv$ 
the unit object, and by $\sigma$ the symmetry morphisms.

A functor between symmetric monoidal categories will be called a {\it SM 
functor} 
if it respects the SM structure.

Given an object $V$ in a SM category, we will denote by $$coev: \triv \to V 
\otimes V^*, \;\; ev: V^* \otimes V \to \triv$$ the coevaluation and evaluation 
maps for $V$. We will also denote by $\gl(V):= V \otimes V^*$ the internal 
endomorphism space with the obvious Lie algebra structure on it. The object $V$ 
is then a module over the Lie algebra $\gl(V)$; we denote the action by $act: 
\gl(V) \otimes V \to V$.
\VeraA{ For two functors $F,G$ we write $F\vdash G$ if $F$ is left adjoint of $G$.}
\subsection{The periplectic Lie superalgebra}\label{ssec:notn:periplectic}
\subsubsection{Definition of \VeraA{the} periplectic Lie superalgebra}\label{sssec:def_periplectic}
Let $n \in \Z_{>0}$, and let $V_n$ be an $(n|n)$-dimensional vector superspace 
equipped with a non-degenerate odd symmetric form
\begin{eqnarray}
\label{beta}
\beta:V_n\otimes V_n\to\mathbb C,\quad \beta(v,w)=\beta(w,v), 
\quad\text{and}\quad 
\beta(v,w)=0\,\,\text{if}\,\,p(v)=p(w).
\end{eqnarray}

Then $\operatorname{End}_{\mathbb C}(V_n)$ inherits the structure of a vector 
superspace from $V_n$. We denote by $\mathfrak{p}(n)$ the Lie superalgebra of 
all 
$X\in\operatorname{End}_{\mathbb C}(V_n)$ preserving $\beta$, i.e. satisfying  
$$\beta(Xv,w)+(-1)^{p({X})p(v)}\beta(v,Xw)=0.$$

\begin{remark}\label{rmk:basis}
Choosing dual bases $v_1, v_2, \ldots, v_n$ in $V_{\bar{0}, n}$ and $v_{1'}, 
v_{2'},\ldots v_{n'}$ in $V_{\bar{1}, n}$, we can write the matrix of $X\in 
\mathfrak{p}(n)$ as
$\left(\begin{smallmatrix}A&B\\C&-A^t\end{smallmatrix}\right)$
where $A,B,C$ are $n\times n$ matrices such that $B^t=B,\, C^t=-C$.
\end{remark}

We will also use the triangular decomposition 
${\p(n)} \cong 
{\p(n)}_{-1} \oplus {\p(n)}_0 \oplus {\p(n)}_1$ where $${\p(n)}_0 \cong \gl(n), \;\; 
{\p(n)}_{-1} 
\cong \Pi\wedge^2(\C^n)^*, \;\; {\p(n)}_1 \cong \Pi S^2\C^{n}.$$
Then the action of $\p(n)_{\pm 1}$ on any $\p(n)$-module is ${\p(n)}_0$-equivariant.

\subsubsection{Weights for the periplectic 
superalgebra}\label{sssec:notn:weight_p_n}
The integral weight lattice for $\mathfrak{p}(n)$ will be $span_{\mathbb 
Z}\{\eps_i\}_{i=1}^n$.

\begin{itemize}[label={$\star$}]
 \item We fix a set of simple roots $\eps_2 - \eps_{1}, 
\ldots, \eps_n-\eps_{n-1},\VeraA{-\eps_{n-1}-\eps_n}$, the \VeraA{last} root is odd and all others are even.

Hence the dominant integral weights will be given by $\lambda = \sum_i 
\lambda_i \eps_i$, where $\lambda_1 \leq \lambda_2 \leq \ldots\leq\lambda_n$.
\item We fix an order on the weights of $\p(n)$: for weights $\mu, \lambda$, we 
say that $\mu \geq \lambda$ if $\mu_i \leq \lambda_i$ for each $i$.

\begin{remark}
 It was shown in \cite[Section 3.3]{BDE:periplectic} that if $\leq$ corresponds to a 
highest-weight structure on the category of finite-dimensional representations 
of $\p(n)$.\VeraA{ Note that in the cited paper we use slightly different set of simple roots
 $-\eps_1-\eps_2,\eps_1 - \eps_{2}, \ldots, \eps_{n-1}-\eps_{n}$.} 
\end{remark}

\item
The simple finite-dimensional representation of $\p(n)$ corresponding to the 
weight $\lambda$ whose highest weight vector is {\it even} will be denoted by 
$L_n(\lambda)$.
\begin{example}
 Let $n\geq 2$. The natural representation $V_n$ of $\p(n)$ has highest weight $-\eps_1$, 
with odd highest-weight vector; hence $V_n \cong \Pi L_n(-\eps_1)$. The 
representation $\bigwedge^2 V_n$ has highest weight $-2\eps_1$, and the 
representation $S^2 V_n$ has highest weight $-\eps_1 - \eps_2$; both have even 
highest weight vectors, so $$\wedge^2 V_n \twoheadrightarrow L_n(-2\eps_1), \;\; L_n(-\eps_1 - \eps_2) \hookrightarrow S^2 V_n .$$
\end{example}
\item Set $\rho^{(n)} = 
\sum_{i=1}^n (i-1)\eps_i$, and for any weight 
$\lambda$, denote $$\bar{\lambda}  = \lambda+\rho^{(n)} .$$

\item We will associate to $\lambda$ a weight diagram $d_{\lambda}$, defined as 
a labeling of the integer line by symbols $\bullet$ (``black ball'') and 
$\circ$ 
(``empty'') such that such that $j$ has label $\bullet$ if $j \in 
\{\bar{\lambda_i}\,|\,i=1, 2, \ldots \}$, and label 
$\circ$ otherwise.

\item We denote $\abs{\lambda} := - \sum_i {\lambda_i}$ and $\kappa(\lambda) = \sum_{i} (-1)^{\bar{\lambda}_i}$.
 
\end{itemize}

\subsubsection{Representations of 
\texorpdfstring{$\p(n)$}{p(n)}}\label{sssec:notn:repr_periplectic}

We denote by $\F_n$ the category of 
finite-dimensional representations of
$\p(n)$ whose restriction to $\p(n)_{\bar 0} \cong 
\mathfrak{gl}(n)$ integrates to an action of $GL(n)$.

By definition, the morphisms in $\F_n$ will be {\it 
grading-preserving} 
$\p(n)$-morphisms,
i.e., $\operatorname{Hom}_{\F_n}(X,Y)$ is a vector space and not a 
vector superspace. This is important in order to ensure that the category
$\F_n$ be abelian.

The category $\F_n$ is not semisimple. In fact, this 
category is a highest-weight category, 
having simple, standard, costandard, and projective modules (these are also 
injective and tilting, per \cite{BoeKujawaNakano:complexity}). Given a simple module $L_n(\lambda)$ in 
$\F_n$, we denote the corresponding standard, costandard, and projective 
modules by $\Delta_n(\lambda)$, $\nabla_n(\lambda)$, $P_n(\lambda)$ 
respectively.

\subsubsection{Tensor Casimir and translation functors}

Consider the following natural endomorphism ${\Omega}^{(n)}$ of the endofunctor $(-) 
\otimes V_n$ on $\mathcal{F}_n$.

\VeraA{ Note that $\p_n$ is the set of fixed points of the involutive automorphism $\sigma$ of $\gl(V_n)$.
We consider the $\p_n$-equivariant decomposition:
$$\gl(V_n) \cong\p_n \oplus \p_n^*$$
where $\p_n^*$ is the eigenspace of $\sigma$ with eigenvalue $-1$.
Both $\p_n$ and $\p_n^*$ are maximal isotropic subspaces with respect to the invariant symmetric form on $\gl(V_n)$ and hence this form defines
a non-degenerated pairing $\p_n^*\otimes\p_n\to\C$.}

We begin by taking the orthogonal $\p_n$-equivariant decomposition
$$\gl(V_n) \cong\p_n \oplus \p_n^*$$ with respect to 
the form $$tr:=ev\circ \sigma_{V_n, V_n^*}: \gl(V_n) \cong V_n \otimes V_n^* 
\to 
\C.$$

\begin{definition}[Tensor Casimir]\label{def:Casimir}
For any $M \in \mathcal{F}_n$, let ${\Omega}_M$ be the composition
$$ V_n \otimes M \xrightarrow{\id \otimes coev \otimes \id} V_n \otimes \p(n)^* 
\otimes \p(n) \otimes M \xrightarrow{i_* \otimes \id}  V_n \otimes \gl(V_n) 
\otimes \p(n) \otimes M   \xrightarrow{act \otimes act} 
V_n \otimes M $$ where $i_*: \p(n)^* \to \gl(V_n)$ is the $\p(n)$-equivariant embedding defined above.
\end{definition}
\begin{definition}[Translation functors]\label{def:transl_functors}
 For $k \in \C$, we define a functor 
${\Theta'}^{(n)}_k:\mathcal{F}_n\to\mathcal{F}_n$ as the functor $\Theta^{(n)}= ( - )\otimes V_n$ 
followed by the projection onto the generalized $k$-eigenspace for 
${\Omega}^{(n)}$, i.e. 
\begin{eqnarray}
\label{thetak}
{\Theta'}^{(n)}_k(M):=\bigcup_{m>0}\Ker ({\Omega}^{(n)} 
-k\operatorname{Id})^m_{|_{M\otimes V_n}}
\end{eqnarray}
and set ${\Theta}^{(n)}_k:=\Pi^k{\Theta'}^{(n)}_k$ in case $k\in\mathbb{Z}$ (it was proved in \cite{BDE:periplectic} that $ \forall k \notin \Z, \; {\Theta}^{(n)}_k \cong 0$). 
\end{definition}

We use the following results from \cite{BDE:periplectic} throughout the paper:

\begin{theorem}[See \cite{BDE:periplectic}.]\label{old_thrm:transl_func}
The relations on the translation functors $\Theta_j$, $j \in \Z$ induce a 
representation of the infinite Temperley-Lieb algebra $TL_{\infty}(q=i)$ on the 
Grothendieck ring on $\F_n$. Furthermore, for any $k \in \Z$, ${\Theta}^{(n)}_k \vdash {\Theta}^{(n)}_{k-1}$.
\end{theorem}
The functors $\Theta_k$ are exact, since $ - \otimes V_n$ is an exact functor.
\begin{theorem}[See \cite{BDE:periplectic}.]\label{old_thrm:transl_proj}

Let $P$ be an indecomposable projective module in $\F_n$. Then for any $i$, $\Theta_i P$ is indecomposable 
projective or zero. 
\end{theorem}
For more details on the structure of $\F_n$ we refer the reader 
to 
\cite{BDE:periplectic}.

\subsubsection{Blocks}\label{sssec:blocks}
 There are $2(n+1)$ blocks in the category $\F_n$ has blocks. These blocks are in bijection with the set
$\{ -n, -n+2 ,\ldots, n-2, n\}\times\{+,-\}$.

We have a decomposition 
$$\F_n=\bigoplus_{k\in  \{ -n, -n+2 ,\ldots, n-2, n\}}\left( \F_n \right)^+_k\oplus \bigoplus_{k\in  \{ -n, -n+2 ,\ldots, n-2, n\}}\left( \F_n \right)^-_k,$$
where the functor $\Pi$ (parity change) induces an equivalence $\left( \F_n \right)_k^+ \cong \left( \F_n \right)_k^-$. The block $\left( \F_n \right)_k$ contains all simple modules $L(\lambda)$ with $\kappa(\lambda)=k$.

Hence we may define {\it up-to-pairity blocks} $$\F^k_n := \left( \F_n \right)_k^+ \oplus \left( \F_n \right)_k^-.$$
By abuse of terminology, we will just call these ``blocks'' throughout the paper.

\begin{theorem}[See \cite{BDE:periplectic}.]\label{old_thrm:blocks_action}
Let $i \in \mathbb Z$, $k \in \{ -n, -n +2 , \ldots, n-2, n\}$. Then we have 
$$\Theta_i \VeraA{\F_n^k}\subset \begin{cases} \F_n^{k+2} \,\,\text{if}\,\,i\,\, \text{is odd}\\
\mathcal{F}_n^{k-2} \,\,\text{if}\,\,i\,\, \text{is even}\\
\end{cases}$$
\end{theorem}
Finally, we introduce some notation:

\begin{notation}\label{notn:seq}
Let $I = (i_1, i_2, \ldots, i_k)$ be a sequence with $i_1, i_2, \ldots, i_k \in \Z$.
\begin{enumerate}
 \item We denote by $$\Theta^{(n)}_{I} := \Theta^{(n)}_{i_1} \circ \Theta^{(n)}_{i_2} \circ \ldots \circ \Theta^{(n)}_{i_k}$$ the composition of the corresponding translation functors.
 \item We set
 $$ t(I):= \sum_{s=0}^k (-1)^{i_s+1}.$$
\end{enumerate}
\end{notation}
The following is an immediate corollary of Theorem \ref{old_thrm:blocks_action}:
\begin{corollary}\label{old_cor:blocks_action}
 For any $n \geq 1$,  $l \in \{-n, -(n-2) , \ldots, n-2, n\}$ and any integer sequence $I = (i_1, i_2, \ldots, i_k)$ we have: $$M \in \F_l^n \; \Longrightarrow \Theta^{(n)}_I M \in \F_{l + 2t(I)}^n.$$
\end{corollary}

\subsection{The Duflo-Serganova functor}\label{sec:DS_functor}
\subsubsection{Definition and basic properties}

Let $n \geq 3$, and let $x\in \p(n)$ be an odd element such that $[x,x]=0$. We define the following correspondence of vector superspaces:
\begin{definition}[See \cite{DufloSer:functor}]\label{def:DS}
 Let $M \in \F_n$. We define 
$$ {DS}_x(M) = \quotient{Ker(x \rvert_M)}{Im(x \rvert_M)}.$$ 
\end{definition}
 
The vector superspace $\p_x:=DS_x\p(n)$ is naturally equipped with a Lie superalgebra structure.\VeraA{ One can check by direct computations that
  $\p_x$ is isomorphic to $\p(n-s)$ where $s$ is the rank of $x$. The above correspondence defines an SM-functor $DS_x:\F_n \to \F_{n-s}$, called the {\it Duflo-Serganova functor}}. Such functors were introduced in \cite{DufloSer:functor}.

\section{The Duflo-Serganova functor and the tensor Casimir}
  Let $n \geq 3$, and let $x\in \p(n)_{\bar{1}}$ be such that $[x,x]=0$. Let $s := \rk(x)$.
 
 Definition \ref{def:DS} then gives us a functor $DS_x: \F_n \to \F_{n-s}$.
\begin{lemma}\label{lem:casimir_goes_to_casimir}
We have: $DS_x({\Omega}) = {\Omega}^{(n-s)} DS_n$, where ${\Omega}^{(n)}$ 
is the tensor Casimir for $\p(n)$, and  $\Omega_{n-s}$ is the tensor Casimir for $\p(n-s)$.

That is, for any 
$M \in \mathcal{F}_n$, $DS_x({\Omega}^{(n)}_M) = {\Omega}^{(n-s)}_{DS_x(M)}$ as endomorphisms of 
$V_{n-s} \otimes DS_x(M)$.
\end{lemma}
\begin{proof}
 This follows directly from the definition of the tensor Casimir (Definition \ref{def:Casimir}), as well as the fact that $DS_x$ is a SM 
functor.
\end{proof}

\begin{corollary}\label{cor:DS_translation_commute}
 The functor $DS_x$ commutes with translation functors, that is we have a natural isomorphism of functors $$DS_x {\Theta}^{(n)}_k \stackrel{\sim}{\longrightarrow}
\Theta^{(n-s)}_k DS_x$$ for any $k \in \Z$.

\end{corollary}
\begin{proof}
  Recall that $DS_x$ is a SM functor and $DS_x(V_n) \cong V_{n-s}$. Hence we have a natural isomorphism $\eta:DS_x \Theta \longrightarrow \Theta DS_x$, where $\Theta^{(n)} = ( - )\otimes V_n$   is as in Definition \ref{def:transl_functors}.
  
  Now, consider $DS_x\left(\Omega_n\right)$ (the tensor Casimir). By Lemma \ref{lem:casimir_goes_to_casimir}, the diagram below commutes:
  $$ \xymatrix{&DS_x\Theta^{(n)} \ar[r]^{\eta} \ar[d]_{DS_x\left(\Omega_n\right)} &\Theta^{(n-s)} DS_x \ar[d]^{\Omega_{n-s}DS_x} \\ &DS_x\Theta^{(n)}  \ar[r]^{\eta} &\Theta^{(n-s)} DS_x}.$$

  Hence $\eta$ induces an isomorphism $DS_x \Theta^{(n)}_k \cong \Theta^{(n-s)}_k DS_x$ for any $k \in \Z$, as required.

\end{proof}

\section{Main result}
Throughout this section, we will work with functors $DS_x: \F_n \to \F_{n-2}$ where $x \in  \in \mathfrak{p}(n)_{-1}$ has rank $2$ and satisfies $[x,x]=0$. Here $\mathfrak{p}(n)_{-1}$ is as in Section \ref{sssec:def_periplectic}.

\begin{theorem}\label{thrm:KW}
 We have $\sdim M =0$ if $M \in \F^k_n$, $ k\neq 0, \pm 1$.
\end{theorem}

\begin{proof}
Let $x \in \mathfrak{p}(n)_{-1}$ such that $[x,x]=0$, $\rk(x)=2$. 

We prove the statement by induction on $n$. 

The base case $n=1$ is tautological (in that case, there are no non-zero blocks except $\F^1_{\pm 1}$); in the base case $n=2$, it is enough to check for simple $M \in \F^2_{\pm 2}$.
The blocks $\F^2_{\pm 2}$ are so-called typical blocks, hence all the simple objects in $\F^2_{\pm 2}$ are costandard ($L(\mu) = \nabla(\mu)$), and hence have dimension zero (cf. \cite{BDE:periplectic}). 

Next, for the inductive step, let $n \geq 3$, and assume that our statement holds for $n-2$. 

Let $M \in \F^k_n$, $k \notin \{0, 1\}$. We use Proposition \ref{prop:DS_blocks} below, which states that $DS_x$ preserves blocks, to show that $DS_x(M)$ lies in the corresponding block $\F_{n-2}^k$ (if $n=k$, then $DS_x(M)=0$). 

The fact that the functor $DS_x: \F_n \to \F_{n-2}$ is SM, and hence preserves dimensions, allows us to use the inductive assumption to show that $ \sdim DS_x(M)=0$, and hence $\sdim M=0$, as required.

\end{proof}

  \begin{proposition}\label{prop:DS_blocks}
  Let $x \in \mathfrak{p}(n)_{-1}$ such that $[x,x]=0$, $\rk(x)=2$. Let $M \in \F^k_n$. Then $DS_x (M) \in \F_{n-2}^k$ (if $k = \pm n$, then $DS_x (M)=0$).
  \end{proposition}
\begin{proof}
The proposition is proved in several steps.
 \begin{enumerate}
  \item\label{itm:blocks1} We prove that $DS_x$ commutes with translation functors. This is done in Lemma \ref{cor:DS_translation_commute}.
  \item\label{itm:blocks2} We prove that it is possible to translate any simple module into any typical block; that is, for any simple $L \in \F_n$, 
  \begin{itemize}
   \item There exists an integer sequence $I =(i_1, i_2, \ldots, i_s)$ such that $$\Theta^{(n)}_I L \in \F_{n}^n \;\text{ and } \; \Theta^{(n)}_I L \neq 0.$$ 
   \item There exists an integer sequence$I =(i_1, i_2, \ldots, i_s)$ such that $$\Theta^{(n)}_I L \in \F_{-n}^n \;\text{ and } \; \Theta^{(n)}_I L \neq 0.$$ 
  \end{itemize}
This is proved in Lemma \ref{lem:transl_into_typical}.

  \item\label{itm:blocks3} We prove \VeraA{the statement in the case} $k= \pm n$. In that case, we need to show that $DS_x(M) =0$ for any $M \in \F^{\pm n}_n$. 
  Recall from \cite[Remark 9.1.3]{BDE:periplectic} that any simple module in such a block is also costandard (that is, a thin Kac module in the terminology of \cite{BDE:periplectic}). Hence it is a free $U(\mathfrak{p}(n)_{-1})$-module. Any $\p(n)$-module in $\F^{\pm n}_n$ has a finite filtration with simple subquotients, hence is a free $U(\mathfrak{p}(n)_{-1})$-module. This implies that $DS_x(M)=0$ for any $M \in \F^{\pm n}_n$.
\item Consider a simple module $L \in \F^k_n$, and a simple subquotient $L'$ of $DS_x(L)$. Let $l$ be such that $L' \in \F_{n-2}^l$.

We will show that $k=l$.

Assume $l < k$. Recall that $k \equiv l \, (\mathrm{mod} \, 2)$, as explained in Section \ref{sssec:blocks}.

By \eqref{itm:blocks2}, there exists an integer sequence $I :=(i_1, i_2, \ldots, i_s)$ such that the translation functor $$\Theta^{(n-2)}_I :=\Theta^{(n-2)}_{i_1} \circ \Theta^{(n-2)}_{i_2} \circ \ldots \circ \Theta^{(n-2)}_{i_s}$$ on $\F_{n-2}$ satisfies: $$\Theta^{(n-2)}_I DS_x L' \neq 0 \; \text{ and } DS_x L' \in \F_{n-2}^{n-2}.$$ 

Furthermore, by Corollary \ref{old_cor:blocks_action}, we have: $t(I)=(n-2-l)/2 $. 

Let $$ \Theta^{(n)}_I :=\Theta^{(n)}_{i_1} \circ \Theta^{(n)}_{i_2} \circ \ldots \circ \Theta^{(n)}_{i_s}$$ be the corresponding translation functor on $\F_{n}$.

By \eqref{itm:blocks1}, we have an isomorphism $$\Theta^{(n-2)}_I DS_x L \cong DS_x \Theta^{(n)}_I L.$$ 

By our construction, this object has a non-zero direct summand in the typical block $\F_{n-2}^{n-2}$.

Let us show that $\Theta^{(n)}_I L \in \F_n^n$. Indeed, we may apply Corollary \ref{old_thrm:blocks_action} to get: $$\Theta^{(n)}_I L = \Theta^{(n)}_{i_1} \circ \Theta^{(n)}_{i_2} \circ \ldots \circ \Theta^{(n)}_{i_s} L \in \F_n^{k+2t(I)}$$ We already computed that $t(I)=(n-2-l)/2$ hence $\F_n^{k+2t(I)} = \F_n^{n+(k-l)-2}$. Recall that $k>l$ and they have the same parity, so $n+(k-l)-2 \geq n$. Now, if $n+(k-l)-2 >n$, then $\F_n^{n+(k-l)-2} = 0$, so we can just say that $\Theta^{(n)}_I L \in \F_n^n$.

This implies that $DS_x \Theta^{(n)}_I L=0$ (by \eqref{itm:blocks3}), and hence cannot have a non-zero direct summand in the typical block $\F_{n-2}^{n-2}$. Thus we obtained a contradiction.

A similar proof shows that we cannot have $l >k$: in that case, we to translate to typical block $\F_{n-2}^{-(n-2)}$.

 \end{enumerate}
\end{proof}

\begin{lemma}\label{lem:transl_into_typical}
 For any simple module $L \in \F_n$, 
 \begin{enumerate}
  \item\label{itm:typical1} There exists a composition of translation functors $\Theta_I$ where $I = (i_1 , i_2,\ldots, i_s)$ is an integer sequence, such that $\Theta^{(n)}_I L \neq 0$ sits in the typical block $\F_n^{n}$.
  \item\label{itm:typical2} There exists a composition of translation functors $\Theta_I$ where $I = (i_1 , i_2,\ldots, i_s)$ is an integer sequence, such that $\Theta^{(n)}_I L \neq 0$ sits in the typical block $\F_n^{-n}$.
 \end{enumerate}

\end{lemma}

\begin{proof}
 We use the results of \cite{BDE:periplectic} on the action of translation functors. In particular, we use the description action of translation functors on projective modules given in \cite[Section 7.2]{BDE:periplectic} as well as the adjuction $\Theta^{(n)}_i \vdash \Theta^{(n)}_{i-1}$ for any $i \in \Z$. 
 
 We first prove \eqref{itm:typical1}.
 
 Let $\lambda$ be the highest weight of $L$ (hence $L = L_n (\lambda)$), let $P = P_n(\lambda)$ be the projective cover of $L$.
 
Fix a typical weight $\mu$ of $\p(n)$ with $\mu_i <\lambda_i$ and $\mu_i \in 2\Z$ for all $i$. Such a weight clearly exists: take for example $\mu_n \in \Z $ such that $\mu_n- n \in 2\Z$ and $ \mu_n < \lambda_1$. Set $\mu_k := \mu_n - (n-k)$ for any $k= 1, 2, \ldots, n-1$. Then $\mu_k < \mu_n < \lambda_1 < \lambda_k$ for any $k = 1, \ldots, n$, and $$\mu_k + k \equiv \mu_n - n \, (\mathrm{mod} \, 2)$$ which implies $\mu_k + k \in 2\Z$. Hence $\mu$ is a typical weight, and $P' := P_n(\mu)$ sits in $\F_n^{n}$. 

By \cite[Section 7.2]{BDE:periplectic}, we have an integer sequence 
$J = (j_1 , j_2,\ldots, j_s)$ such that $\Theta^{(n)}_J P' = P$.

Set $I = (j_s -1 , j_{s-1} -1,\ldots,j_2-1,  j_1-1)$. Then $\Theta^{(n)}_J \vdash \Theta^{(n)}_I$, and we have:
\VeraA{ 
  $$\Hom_{\p(n)}(P', \Theta^{(n)}_I L)= \Hom_{\p(n)}(\Theta_J^{(n)}P', L) =\Hom_{\p(n)}(P,L)=\C.$$
Therefore   $\Theta^{(n)}_I L$ is a non-trivial quotient of $P'$. This proves \eqref{itm:typical1}.}

Similarly, we prove \eqref{itm:typical2}. Fix a typical weight $\mu$ of $\p(n)$ with $\mu_i <\lambda_i$ and $\mu_i \in 2\Z+1$ for all $i$. Again, such a weight can be constructed very explicitly. Then $P' := P_n(\mu)$ sits in $\F_n^{-n}$, and we can apply exact the same arguments as before. 

\end{proof}

\section{Dual modules and blocks}

\begin{proposition}
 Let $M \in \F_n^k$. Then $M^*$ also lies in block $\F_n^k$.
\end{proposition}

\begin{proof}
 First of all, notice that it is enough to prove the statement for a simple module $M = L(\lambda)  \in \F_n^k$. 
 
 Consider the costandard module $\nabla(\lambda)$ having $L(\lambda)$ as its socle. This module is indecomposable, so $ \nabla(\lambda)\in \F_n^k$. Consider the dual module $\nabla(\lambda)^*$. This is also an indecomposable costandard module, with cosocle $L(\lambda)^*$, so it is enough to check that $\nabla(\lambda)^*\in \F_n^k$ as well. Now, by \cite[Lemma 3.6.1]{BDE:periplectic}, $\nabla(\lambda)^* \cong \nabla(\mu)$, where $\mu + \rho = -w_0(\lam + \rho)$, where $w_0$ is the longest element in the Weyl group. That is, $d_{\mu}$ is obtained from $d_{\lam}$ by reflecting the diagram with respect to zero. 
 
 Hence $\kappa(\lam) = \kappa(\mu)$, and so $L(\mu) \in \F_n^k$.
\end{proof}

\comment{\Inna{In fact, if we denote by $\lam^*$ the highest weight of $L(\lam)^*$, we have: the beginning positions of consecutive black sequences (``black clusters'') in $d_{\lam}$ and $d_{\lam^*}$ are obtained from one another by reflecting with respect to zero. The number of such clusters in both cases is the same, due to the lemma below $+$ $\Theta_i L(\lam) \neq 0$ only if $i$ is the beginning of a black cluster, I believe  }}

\begin{proposition}
 There exists a natural isomorphism $$ \Pi\Theta^{(n)}_{-k} (-)^* \stackrel{\sim}{\longrightarrow}  \left(\Theta^{(n)}_{k} (-)\right)^* .$$ 
\end{proposition}

\begin{proof}

Consider the functor $- \otimes V_n:\F_n \to \F_n$, $M \mapsto M \otimes V_n$. We have natural isomorphisms
$$M^* \otimes V_n \otimes \Pi\C \xrightarrow{\id \otimes \eta^{-1}} M^* \otimes V_n^* \xrightarrow{\sigma_{M^*, V_n^*}} V_n^* \otimes M^*\cong (M \otimes V_n)^*$$ where $\eta: V_n^* \to V_n \otimes \Pi\C$ is the isomorphism defined by the odd bilinear form on $V_n$. 

Consider the tensor Casimir $\Omega_M: M \otimes V_n \to M \otimes V_n$ as in Definition \ref{def:Casimir}. Choosing dual bases $\{X_i\}_i, \{X^i\}_i$ in $\p(n)$ and $\p(n)^* \subset \gl(n|n)$, we can write $\Omega_M = \sum_i X_i\rvert_M \otimes X^i\rvert_{V_n}$.

Denote by $$\Omega^*_M: V_n^* \otimes M^* \to V_n^* \otimes M^*$$ the dual map. Then for any homogeneous $u \in V_n^*, f \in M^*$, we have:$$\Omega_M^*(u \otimes f) = \sum_i (-1)^{p(X_i)p(X^i)}(-1)^{p(X^i)p(f)}\left(X^i\rvert_{V_n} \right)^*(u) \otimes \left(X_i\rvert_M \right)^*(f).$$

We now construct the commutative diagram $$ \xymatrix{&{V_n^* \otimes M^*} \ar[rr]^{\Omega_M^*} &{} &{ V_n^* \otimes M^*} \ar[d]^{\sigma_{M^*, V_n^*}^{-1}}\\ &{M^* \otimes V_n^*} \ar[u]^{\sigma_{M^*, V_n^*}} \ar[rr] &{} &{M^* \otimes V_n^*} \ar[d]^{\id \otimes \eta} \\ &{M^* \otimes V_n \otimes \Pi\C} \ar[u]^{\id \otimes \eta^{-1}} \ar[rr] &{}&{M^* \otimes V_n \otimes \Pi\C}}$$
 and we compute the lower two horizontal arrows.
 
 We begin with the horizontal arrow $\phi: M^* \otimes V_n^* \to M^* \otimes V_n^*$. 
 
 By definition, $\phi =\sigma_{M^*, V_n^*}^{-1} \circ \Omega_M^* \circ \sigma_{M^*, V_n^*}$.

 For any homogeneous $f \in M^*, u \in V_n^*$, applying the map $\Omega_M^* \circ \sigma_{M^*, V_n^*}$ to the element $f \otimes u$ we get:
 \begin{align*}
  &\Omega_M^*\left((-1)^{p(f)p(u)} u \otimes f\right) =\\&=\sum_i (-1)^{p(f)p(u)} (-1)^{p(X_i)p(X^i)}(-1)^{p(X^i)p(u)+1} (-1)^{{p(X^i)p(f)}} X^i.u \otimes (-1)^{p(X^i)p(f)+1} X_i.f =\\&=\sum_i (-1)^{p(X_i)p(X^i)+p(f) p(u)+{p(X^i)p(u)}} X^i.u \otimes X_i.f
 \end{align*}

 Hence 
 \begin{align*}
  \phi(f \otimes u) &= \sum_i (-1)^{p(X_i)p(X^i)+p(f) p(u)+{p(X_i)p(u)}} (-1)^{(p(X_i)+p(f))(p(X^i)+p(u))} X_i.f \otimes X^i.u =\\&= \sum_i (-1)^{{p(X^i)p(f)}} X_i.f \otimes X^i.u  
  \end{align*}
 
 Next, we compute the horizontal arrow $\phi': M^* \otimes V_n \otimes \Pi\C \to M^* \otimes V_n \otimes \Pi\C$.
 
The elements $X^i \in \gl(n|n)$ satisfy the following property (cf. \cite[Proof of Proposition 4.4.1]{BDE:periplectic}):
 $$ \xymatrix{&V_n^* \ar[r]^{X^i} \ar[d]^{\eta} &V_n^* \ar[d]^{\eta} \\ &{V_n \otimes \Pi \C} \ar[r]_{-X^i} &{V_n \otimes \Pi \C}}$$
 
Given homogeneous $f \in M^*, v \in V_n \otimes \Pi\C$, we have $$\phi'(f \otimes v) = -(-1)^{{p(X^i)p(f)}} X_i.f \otimes X^i.v$$

Hence $\phi'=-\Pi\Omega_{M^*}$. Thus the natural isomorphism $\sigma_{M*, V_n^*} \circ (\id \otimes \eta^{-1})$ establishes a natural isomorphism between the eigenspace of $\Omega_M^*$ corresponding to eigenvalue $k$ and the eigenspace (shifted by $\Pi$) of $\Omega_{M^*}=-\Pi\phi'$ corresponding to eigenvalue $(-k)$. This implies the statement of the proposition.
\end{proof}
\begin{example}
 Let $n\geq 2$ and set $M=V_n$. Then $V_n^* \cong \Pi V_n$ and $\Theta_{-1}V_n \cong \wedge^2 V_n$ and hence $$\left(\Theta_{-1}V_n\right)^* \cong S^2 V_n \cong \Pi \Theta_1 (\Pi V_n) \cong \Pi \Theta_1 V_n^*.$$
\end{example}

\end{document}